\documentclass{article}
\usepackage{amssymb,amsmath,bm,geometry,graphics,graphicx,url,color}

\def\no{\noindent}
\def\pmatrix{\left(\begin{array}}
\def\endpmatrix{\end{array}\right)}

\newtheorem{theo}{Theorem}
\newtheorem{lem}{Lemma}

\newtheorem{rem}{Remark}
\newtheorem{defi}{Definition}

\newcommand{\norm}[1]{\left\Vert#1\right\Vert}

\title{Functionally-fitted energy-preserving integrators\\ for Poisson systems}

\author{Bin Wang\,
\footnote{School of Mathematical Sciences, Qufu Normal University,
Qufu 273165, P.R. China; Mathematisches Institut, University of
T\"{u}bingen, Auf der Morgenstelle 10, 72076 T\"{u}bingen, Germany.
The research is supported in part by the Alexander von Humboldt
Foundation and by the Natural Science Foundation of Shandong
Province (Outstanding Youth Foundation) under Grant ZR2017JL003.
E-mail:~{\tt wang@na.uni-tuebingen.de} } \and Xinyuan
Wu\thanks{School of Mathematical Sciences, Qufu Normal University,
Qufu 273165, P.R. China; Department of Mathematics, Nanjing
University, Nanjing 210093, P.R. China. The research is supported in
part by the National Natural Science Foundation of China under Grant
11671200. E-mail:~{\tt xywu@nju.edu.cn}} }

\begin{document}
\maketitle

\begin{abstract} In this paper, a new class of energy-preserving integrators is
proposed and analysed for Poisson systems by using
functionally-fitted technology. The integrators exactly preserve
energy and have arbitrarily high order. It is shown that the
proposed approach allows us to obtain the energy-preserving methods
derived in BIT 51 (2011) by Cohen and Hairer and in J. Comput. Appl.
Math. 236 (2012) by Brugnano et al. for Poisson systems.
Furthermore, we study the sufficient conditions that ensure the
existence of a unique solution and discuss the order of the new
energy-preserving integrators.

\medskip
\no{\bf Keywords:}  Poisson systems,  energy preservation,
functionally-fitted integrators

\medskip
\no{\bf MSC:}65P10, 65L05

\end{abstract}

\section{Introduction}
In this paper, we deal with the efficient numerical integrators for
solving   the Poisson systems  (non-canonical Hamiltonian systems)
\begin{equation}\label{poisson sys}
 \dot{y}=B(y) \nabla H(y),\quad y(0)=y_{0}\in\mathbb{R}^{d},\quad
t\in[0,T],
\end{equation}
where $B(y)$ is a skew-symmetric matrix  which is not required to
satisfy  the Jacobi identity. It is well known that the energy  $
H(y)$ is preserved along the exact solution of \eqref{poisson sys},
since
$$ \frac{d H(y)}{dt} =\nabla H(y)^{\intercal}\dot{y}=\nabla H(y)^{\intercal} B(y) \nabla H(y)=0.$$
Numerical integrators that preserve  $ H(y)$ are usually called
energy-preserving (EP) integrators, and the aim of this paper is to
formulate and analyse novel EP integrators for  efficiently solving
Poisson systems.


When the matrix  $B(y)$ is independent of $y$, the system
\eqref{poisson sys} becomes a canonical Hamiltonian system. There
have been a lot of studies on numerical methods   for this system,
and  the reader is referred to
\cite{Feng-book,Hairer09,Hairer16,hairer2006,JCP2017_Mei_Wu,wang-2016,wang2014,wu2017-JCAM,wubook2015,wu2013-book}
and references therein. For  canonical  Hamiltonian systems,  EP
methods are an important and efficient kind of methods and many
various of  EP methods have been derived and studied in the past few
decades, such as the average vector field (AVF) method (see, e.g.
   \cite{Celledoni2010,Celledoni14,Quispel08}), discrete gradient
methods  (see, e.g.  \cite{McLachlan14,McLachlan99}) ,  Hamiltonian
Boundary Value Methods (HBVMs) (see, e.g.
\cite{Brugnano2010,Brugnano2012b}) ,   EP collocation methods  (see,
e.g. \cite{Hairer2010} )
 and exponential/trigonometric EP
   methods  (see, e.g. \cite{Li_Wu(sci2016),Miyatake2014,wang2012-1,17-new,wu2013-JCP}).

 Among these  EP methods for solving $\dot{y}=J \nabla H(y)$,  the
AVF method  has the simplest form, which was given by Quispel and
McLaren \cite{Quispel08} as follows
\begin{equation}\label{AVF}
 y_{1} =y_0+h  \int_{0}^1  J \nabla H(y_0+\sigma
(y_1-y_0)) d \sigma.
\end{equation}
 Hairer extended this second-order
method to  higher order schemes  by introducing continuous stage
Runge--Kutta
 methods \cite{Hairer2010}. However,  because
the dependence of the matrix $B(y)$  should be discretised in a
different manner, Poisson systems usually require an additional
technique. Therefore,  the
 novel EP methods which are specially designed and analysed  for Poisson systems are  necessary.
McLachlan et al. \cite{McLachlan99} discussed DG methods for various
kinds of ODEs including Poisson systems.
  Cohen and Hairer in  \cite{Cohen-2011} succeeded
in constructing arbitrary high-order EP schemes for Poisson systems
and the following second-order EP scheme for \eqref{poisson sys} was
derived
\begin{equation}\label{EXAVF}
 y_{1} =y_0+h  B\big(\frac{y_1+y_0}{2}\big)\int_{0}^1   \nabla H(y_0+\sigma
(y_1-y_0)) d \sigma.
\end{equation}
 Following the ideas
of  HBVMs,  Brugnano et al. gave  an alternative derivation of such
methods and presented a new proof of their orders in
\cite{Brugnano2012}.
 EP exponentially-fitted
integrators for Poisson systems were researched by Miyatake
\cite{Miyatake2015}. Based on discrete gradients, Dahlby et al.
\cite{Dahlby2012} constructed  useful methods that simultaneously
preserve  several invariants in systems of type \eqref{poisson sys}.

On the other hand,  the functionally-fitted (FF)  technology is a
popular approach to constructing  effective and efficient  methods
in scientific computing.  An FF method is generally derived
 by requiring it to integrate members of a given
finite-dimensional function space $X$ exactly.  The corresponding
methods are  called as  trigonometrically-fitted (TF) or
exponentially-fitted (EF) methods if $X$ is generated by
  trigonometrical or exponential functions. Using
FF/TF/EF technology, many efficient methods have been constructed
for canonical Hamiltonian systems including the symplectic methods
(see, e.g.
\cite{Calvo2009,Calvo2010a,Calvo2010b,Franco2007,VandenB2003,VandeVyver2006,wang2017-Cal,wu-2012-BIT})
 and   EP methods (see, e.g.
\cite{Li_Wu(na2016),Miyatake2014}). This technology has also been
used successfully for Poisson systems in \cite{Miyatake2015} and
 second- and fourth-order schemes were derived.
 In this  paper,  using
the functionally-fitted technology, we will   design and analyse
novel EP integrators for Poisson systems. The new integrators can be
of arbitrary order in a  routine and  convenient manner,   and
different EP schemes can be obtained by considering different
function spaces. It will be shown that  choosing a special function
space allows us to obtain the
 EP  schemes  given by Cohen and Hairer \cite{Cohen-2011}
and Brugnano et al. \cite{Brugnano2012}.

This paper is organised as follows. In Section \ref{fomulation of
methods}, we derive the  EP   integrators  for Poisson systems.
Section \ref{implementations} is devoted to the implementation
issues.  The existence and uniqueness of the integrators
 are studied in Section \ref{existence} and their algebraic orders
are discussed in Section \ref{sec order}.
 In Section \ref{methods},  two second-order
EP   schemes are presented as illustrative examples. Numerical
experiments are implemented  in  Section \ref{Numerical
experiments}, where we consider the Euler equation.   The last
section includes some
 conclusions.

\section{Functionally-fitted EP integrators}\label{fomulation of methods}
In this paper, we define a function space
$Y$=span$\left\{\varphi_{0}(t),\ldots,\varphi_{r-1}(t)\right\}$ on
$[0,T]$
  by
\begin{equation*}
Y=\left\{w : w(t)=\sum_{i=0}^{r-1}\varphi_{i}(t)W_{i},\ t\in I,\
W_{i}\in\mathbb{R}^{d}\right\},
\end{equation*}
where $\{\varphi_{i}(t)\}_{i=0}^{r-1}$ are  linearly independent on
$[0,T]$ and sufficiently smooth.  We then   consider the following
two finite-dimensional function spaces  $Y$ and $X$
\begin{equation*}\label{FF}
Y=\text{span}\left\{\varphi_{0}(t),\ldots,\varphi_{r-1}(t)\right\},\quad
X=\text{span}\left\{1,\int_{0}^{t}\varphi_{0}(s)ds,\ldots,\int_{0}^{t}\varphi_{r-1}(s)ds\right\}.
\end{equation*}
Choose a stepsize $h$ and define the function spaces $Y_{h}$ and
$X_{h}$ on $[0,1]$ by
\begin{equation}\label{FF-h}
Y_{h}=\text{span}\left\{\tilde{\varphi}_{0}(\tau),\ldots,\tilde{\varphi}_{r-1}(\tau)\right\},\quad
X_{h}=\text{span}\left\{1,\int_{0}^{\tau}\tilde{\varphi}_{0}(s)ds,\ldots,\int_{0}^{\tau}\tilde{\varphi}_{r-1}(s)ds\right\},
\end{equation}
 where $\tilde{\varphi}_{i}(\tau)=\varphi_{i}(\tau h),\ \tau\in[0,1]$
for $i=0,1,\ldots,r-1$. It is noted that    the notation $\tilde{f}
(\tau)$ is referred to as $f(\tau h)$ for all the functions
throughout this paper.

A projection  given in \cite{Li_Wu(na2016)} will be used in this
paper and we    summarise  its definition as follows.
\begin{defi}\label{def projection} (See \cite{Li_Wu(na2016)})
  The definition of
$\mathcal{P}_{h}\tilde{w}$ is given by
\begin{equation}\label{DEF}
\langle
\tilde{v}(\tau),\mathcal{P}_{h}\tilde{w}(\tau)\rangle=\langle
\tilde{v}(\tau),\tilde{w}(\tau)\rangle,\quad \text{for any}\ \
\tilde{v}(\tau)\in Y_{h},
\end{equation}
where $\tilde{w}(\tau)$ be a continuous $\mathbb{R}^{d}$-valued
function on $[0,1]$ and $\mathcal{P}_{h}\tilde{w}(\tau)$ is a
projection of $\tilde{w}$ onto $Y_{h}$.  Here the inner product
$\langle\cdot,\cdot\rangle$ is defined by
\begin{equation*}
\langle \tilde{w}_{1},\tilde{w}_{2}\rangle=\langle
\tilde{w}_{1}(\tau),\tilde{w}_{2}(\tau)\rangle_{\tau}=\int_{0}^{1}\tilde{w}_{1}(\tau)\cdot
\tilde{w}_{2}(\tau)d\tau,
\end{equation*}
where $\tilde{w}_{1}$ and $\tilde{w}_{2}$ are two integrable
functions (scalar-valued or vector-valued)
 on $[0,1]$, and  if they are  both vector-valued functions, `$\cdot$' denotes the entrywise multiplication
operation.
\end{defi}

We also need the  following property of  $\mathcal{P}_{h}$ which has
been proved in \cite{Li_Wu(na2016)}.

\begin{lem}\label{proj lem}(See \cite{Li_Wu(na2016)}) The projection $\mathcal{P}_{h}\tilde{w}$ can be explicitly expressed as
\begin{equation*}
\mathcal{P}_{h}\tilde{w}(\tau)=\langle
P_{\tau,\sigma},\tilde{w}(\sigma)\rangle_{\sigma},
\end{equation*}
where
\begin{equation*}\label{PCOEFF}
P_{\tau,\sigma}=\sum_{i=0}^{r-1}\tilde{\psi}_{i}(\tau)\tilde{\psi}_{i}(\sigma),
\end{equation*}
and $\left\{\tilde{\psi}_{0},\ldots,\tilde{\psi}_{r-1}\right\}$ is a
standard orthonormal basis of $Y_{h}$ under the inner product
$\langle\cdot,\cdot\rangle$.
\end{lem}

 On the basis of  these preliminaries, we first present the
definition of the
  integrators and then show that they exactly preserve the
energy of    Poisson system \eqref{poisson sys}.

\begin{defi}\label{def EDCr}
We consider a function $\tilde{u}(\tau)\in X_{h}$ with
$\tilde{u}(0)=y_{0}$, satisfying
\begin{equation} \label{EDCr}
\begin{aligned} &
\tilde{u}^{\prime}(\tau) =B(\tilde{u}(\tau))\mathcal{P}_{h}\big(
\nabla H(\tilde{u}(\tau))\big),\ \ \ \tau\in[0,1].
\end{aligned}
\end{equation}
The numerical solution after one step is then defined by $
y_{1}=\tilde{u}(1).$ In this paper, we call the integrator   as
functionally-fitted EP (FFEP) integrator.
\end{defi}

\begin{theo}\label{ED}
The FFEP integrator \eqref{EDCr}  exactly preserves  the energy,
i.e.,
$$H(y_{1})= H(y_{0}).$$
\end{theo}
\begin{proof}
Since $\tilde{u}\in X_{h}$, one gets $\tilde{u}'(\tau)\in Y_h$. By
the definition of $\mathcal{P}_{h}$, we have
\begin{equation*}
\int_{0}^{1}\tilde{u}'(\tau)_{i}\Big( \mathcal{P}_{h}\big( \nabla
H(\tilde{u}(\tau))\big)\Big)_{i}d\tau=\int_{0}^{1}\tilde{u}'(\tau)_{i}
\big( \nabla H(\tilde{u}(\tau))\big)_{i}d\tau,\quad i=1,2,\ldots,d,
\end{equation*}
where   $(\cdot)_{i}$ denotes the $i$th entry of a vector. Then, we
obtain
\begin{equation*}
\int_{0}^{1}\tilde{u}'(\tau)^{\intercal}\mathcal{P}_{h}\big( \nabla
H(\tilde{u}(\tau))\big)
d\tau=\int_{0}^{1}\tilde{u}'(\tau)^{\intercal}   \nabla
H(\tilde{u}(\tau)) d\tau.
\end{equation*}
Therefore, one has
\begin{equation*}
\begin{aligned}
&H(y_{1})-H(y_{0})=\int_{0}^{1}\frac{d}{d\tau}H(\tilde{u}(\tau))d\tau\\
 =&h\int_{0}^{1}\tilde{u}'(\tau)^{\intercal}\nabla
H(\tilde{u}(\tau))d\tau=h\int_{0}^{1}\tilde{u}'(\tau)^{\intercal}\mathcal{P}_{h}\big(
\nabla H(\tilde{u}(\tau))\big)d\tau.
\end{aligned}
\end{equation*}
Inserting the integrator \eqref{EDCr} into this formula yields
\begin{equation*}
\begin{aligned}
H(y_{1})-H(y_{0})&=h\int_{0}^{1}\mathcal{P}_{h}\big( \nabla
H(\tilde{u}(\tau))\big)^{\intercal}B(\tilde{u}(\tau))^{\intercal}\mathcal{P}_{h}\big(
\nabla H(\tilde{u}(\tau))\big)d\tau,
\end{aligned}
\end{equation*}
which proves the result by considering that $B(\tilde{u})$ is a
skew-symmetric matrix.
 \end{proof}

\begin{rem}
Consider   $B(y)$ is a constant skew-symmetric matrix, which means
that \eqref{poisson sys} is a  canonical Hamiltonian system. Then
the FFEP integrator \eqref{EDCr} becomes  the functionally-fitted EP
method derived in Li and Wu \cite{Li_Wu(na2016)}. Besides, if $Y_h$
is particularly generated by the shifted Legendre polynomials on
$[0,1]$,
 then the
FFEP integrator \eqref{EDCr} reduces to  the EP collocation
  method  given by Cohen and Hairer
\cite{Hairer2010} and Brugnano et al.  \cite{Brugnano2012}.
\end{rem}
\section{Implementations issues}\label{implementations}
 We choose the generalized Lagrange interpolation
functions $\{\hat{l}_{i}(\tau)\}_{i=1}^{r}\in Y_{h}$ with respect to
$r$ distinct points $\{\hat{d}_{i}\}_{i=1}^{r}\subseteq[0,1]$ as
follows
\begin{equation}\label{Lagrange}
(\hat{l}_{1}(\tau),\ldots,\hat{l}_{r}(\tau))=(\tilde{\varphi}_{0}(\tau),\tilde{\varphi}_{2}(\tau),\ldots,\tilde{\varphi}_{r-1}(\tau))
\left(\begin{array}{cccc}\tilde{\varphi}_{0}(\hat{d}_{1})&\tilde{\varphi}_{1}(\hat{d}_{1})&\ldots&\tilde{\varphi}_{r-1}(\hat{d}_{1})\\
\tilde{\varphi}_{0}(\hat{d}_{2})&\tilde{\varphi}_{1}(\hat{d}_{2})&\ldots&\tilde{\varphi}_{r-1}(\hat{d}_{2})\\ \vdots&\vdots& &\vdots\\
\tilde{\varphi}_{0}(\hat{d}_{r})&\tilde{\varphi}_{1}(\hat{d}_{r})&\ldots&\tilde{\varphi}_{r-1}(\hat{d}_{r})\\\end{array}\right)^{-1}.
\end{equation}
 Then it can be  easily verified    that
$\{\hat{l}_{i}(\tau)\}_{i=1}^{r}$ is another basis of $Y_h$ and
satisfies $\hat{l}_{i}(\hat{d}_{j})=\delta_{ij}.$ Since
$\tilde{u}'(\tau)\in Y_{h}$, $\tilde{u}'(\tau)$ can be expressed by
the basis of $Y_{h}$ as follows
\begin{equation*}
\tilde{u}'(\tau)=\sum_{i=1}^{r}\hat{l}_{i}(\tau)\tilde{u}'(\hat{d}_{i}).
\end{equation*}
By Lemma \ref{proj lem},   the FFEP integrator \eqref{EDCr}  becomes
\begin{equation*}
\begin{aligned} &
\tilde{u}^{\prime}(\tau) =B(\tilde{u}(\tau)) \int_{0}^1
P_{\tau,\sigma} \nabla H(\tilde{u}(\sigma)) d \sigma.
\end{aligned}
\end{equation*}
 Thus,  one arrives
\begin{equation*}
\begin{aligned} &
\tilde{u}'(\tau)=u'(\tau
h)=\sum_{i=1}^{r}\hat{l}_{i}(\tau)B(\tilde{u}(\hat{d}_{i}))
\int_{0}^1 P_{\hat{d}_{i},\sigma} \nabla H(\tilde{u}(\sigma)) d
\sigma.
\end{aligned}
\end{equation*}
By integration we get
\begin{equation*}
\begin{aligned} &
\tilde{u} (\tau) =u(\tau h)=y_0+h\int_{0}^\tau
\sum_{i=1}^{r}\hat{l}_{i}(\alpha)d\alpha B(\tilde{u}(\hat{d}_{i}))
\int_{0}^1 P_{\hat{d}_{i},\sigma} \nabla H(\tilde{u}(\sigma)) d
\sigma. \\
\end{aligned}
\end{equation*}
Denoting   $y_{\sigma}=\tilde{u}(\sigma)$, we obtain  practical
schemes  of  the FFEP integrator \eqref{EDCr} for Poisson system
\eqref{poisson sys}.

\begin{defi}\label{def PEEPCr-2} A
practical   scheme of the FFEP integrator \eqref{EDCr}  for Poisson
system \eqref{poisson sys} is defined  by
\begin{equation}\label{PEEPCr-2}
\left\{\begin{aligned}& y_{\tau} =y_0+h\sum_{i=1}^{r}\int_{0}^{\tau}
\hat{l}_{i}(\alpha)d\alpha B(y_{\hat{d}_{i}}) \int_{0}^1
P_{\hat{d}_{i},\sigma} \nabla H(y_\sigma) d
\sigma, \ \ \ 0< \tau<1,\\
& y_{1} =y_0+h\sum_{i=1}^{r}\int_{0}^1 \hat{l}_{i}(\alpha)d\alpha
B(y_{\hat{d}_{i}})\int_{0}^1 P_{\hat{d}_{i},\sigma} \nabla
H(y_\sigma) d
\sigma.\\
\end{aligned}\right.
\end{equation}
\end{defi}

\begin{rem}
Dnoting  \begin{equation*}
 \begin{aligned}& a_{\tau,i}=\int_{0}^{\tau} \hat{l}_{i}(\alpha)d\alpha,\ \ \ \  X_i=h  B(y_{\hat{d}_{i}}) \int_{0}^1
P_{\hat{d}_{i},\sigma} \nabla H(y_\sigma) d \sigma,\\
\end{aligned}
\end{equation*} and choosing
$\tau=\hat{d}_1,\ldots,\hat{d}_r$ for the first formula of
\eqref{PEEPCr-2},  we   get a linear system of equations for
$X_1,\ldots,X_r$ as
\begin{equation*}
 \begin{aligned}& y_{\hat{d}_j}
=y_0+ \sum_{i=1}^{r} a_{\hat{d}_j,i} X_i, \ \ \ \ j=1,\ldots,r.\\
\end{aligned}
\end{equation*}
Solving this linear system by Cramer's rule yields the results of
$X_i$ for $i=1,\ldots,r.$
Then $y_\sigma$ can be expressed as
$$y_{\sigma} =y_0+ \sum_{i=1}^{r}\int_{0}^{\sigma}
\hat{l}_{i}(\alpha)d\alpha X_i.$$
 Therefore, we need the first formula of
\eqref{PEEPCr-2} only for $\tau=\hat{d}_1,\ldots,\hat{d}_r$ and this
presents a nonlinear system of equations for the unknowns
$y_{\hat{d}_1} ,\ \ldots, y_{\hat{d}_r} $ which can be solved by
iteration.

\end{rem}

\begin{rem}
It is noted that the integrals $\int_{0}^{\tau}
\hat{l}_{i}(\alpha)d\alpha$ and $\int_{0}^{1}
\hat{l}_{i}(\alpha)d\alpha$  can be  calculated exactly. The
integral $\int_{0}^1 P_{\hat{d}_{i},\sigma} \nabla H(y_\sigma) d
\sigma$ appearing  in \eqref{PEEPCr-2} can also be computed  exactly
for some cases such as $\nabla H$ is a polynomial and $Y_h$ is
generated by polynomials, exponential or  trigonometrical functions.
If the integral cannot be directly calculated, it is nature  to
approximate it by  a quadrature formula with nodes $c_i$ and weights
$b_i$ for $i = 1,\ldots , s.$ Then the scheme \eqref{PEEPCr-2}
becomes
\begin{equation*}
\left\{\begin{aligned}& y_{\tau} =y_0+h\sum_{i=1}^{r}\int_{0}^{\tau}
\hat{l}_{i}(\alpha)d\alpha B(y_{\hat{d}_{i}}) \sum_{j=1}^{s} b_j
P_{\hat{d}_{i},c_j} \nabla H(y_{c_j}),\\
& y_{1} =y_0+h\sum_{i=1}^{r}\int_{0}^1 \hat{l}_{i}(\alpha)d\alpha
B(y_{\hat{d}_{i}}) \sum_{j=1}^{s} b_j P_{\hat{d}_{i},c_j} \nabla
H(y_{c_j}).\\
\end{aligned}\right.
\end{equation*}

\end{rem}

%

\section{The existence, uniqueness
 and  smoothness}\label{existence}
It is clear  that the FFEP integrator \eqref{EDCr} fails to be well
defined unless the existence and uniqueness is shown. This  section
is devoted to this point.

  It is assumed in this section that the solution of
\eqref{poisson sys} is bounded by
\begin{equation*}
\begin{aligned}
\bar{B}(y_{0},R)=\left\{y\in\mathbb{R}^{d} : ||y-y_{0}||\leq
R\right\},
\end{aligned}
\end{equation*}
where  $R$ is  a positive constant.  The $n$th-order derivatives of
$\nabla H(y)$ and $B(y)$ at $y$ are denoted by $\nabla H^{(n)}(y)$
and $B^{(n)}(y)$, respectively.  Besides, it has been shown  in
\cite{Li_Wu(na2016)} that $P_{\tau,\sigma}$ is a smooth function of
$h$. Under this background, we assume that
\begin{equation*}
\begin{aligned}
&A_{n}=\max_{\tau,\sigma,h\in[0,1]}\mid \frac{\partial^n P_{\tau,\sigma}}{\partial h^n}\mid,\\
&C_{n}=\max_{y\in \bar{B}(y_{0},R)}||B^{(n)}(y)||,\\
&D_{n}=\max_{y\in \bar{B}(y_{0},R)}||\nabla H^{(n)}(y)||,\quad
n=0,1,\ldots.
\end{aligned}
\end{equation*}

\begin{theo}\label{eus}
Under the above assumptions,  the  FFEP integrator \eqref{EDCr} has
a unique solution $\tilde{u}(\tau)$  if the stepsize $h$ satisfies
\begin{equation}\label{cond2} 0\leq
h\leq \delta
<\min\left\{\frac{1}{A_0C_0D_{1}+A_0C_1D_{0}},\frac{R}{A_0C_0D_{0}},1\right\}.
\end{equation}
Moreover,  $\tilde{u}(\tau)$ is smoothly dependent  on  $h$.
\end{theo}
\begin{proof}By Lemma \ref{proj lem}, the FFEP integrator \eqref{EDCr} can be rewritten as
\begin{equation*}
\begin{aligned} &
\tilde{u}^{\prime}(\tau) =B(\tilde{u}(\tau)) \int_{0}^1
P_{\tau,\sigma} \nabla H(\tilde{u}(\sigma)) d \sigma.
\end{aligned}
\end{equation*}
By integration we arrive at
\begin{equation*}\label{method inte}
\begin{aligned} &
\tilde{u} (\tau) =y_0+h\int_{0}^\tau B(\tilde{u}(\alpha)) \int_{0}^1
P_{\alpha,\sigma} \nabla H(\tilde{u}(\sigma)) d \sigma d\alpha.
\end{aligned}
\end{equation*}
Based on this formula, we  get a function series
$\{\tilde{u}_{n}(\tau)\}_{n=0}^{\infty}$ by the following  recursive
definition
\begin{equation}\label{recur}
\tilde{u}_{n+1}(\tau)=  y_0+h\int_{0}^1\big(\int_{0}^\tau
B(\tilde{u}_n(\alpha))
 P_{\alpha,\sigma}d\alpha\big) \nabla H(\tilde{u}_n(\sigma)) d \sigma
, \quad n=0,1,\ldots,
\end{equation}
which will be shown to be  uniformly convergent by proving the
uniform convergence of the infinite series
$\sum_{n=0}^{\infty}(\tilde{u}_{n+1}(\tau)-\tilde{u}_{n}(\tau)).$
Then the integrator \eqref{EDCr}  has  a solution
  $\lim\limits_{n\to\infty}\tilde{u}_{n}(\tau)$.

 We now  prove  the uniform convergence of
$\sum_{n=0}^{\infty}(\tilde{u}_{n+1}(\tau)-\tilde{u}_{n}(\tau)).$
Firstly, it is clear that $||\tilde{u}_{0}(\tau)-y_{0}||=0\leq R$.
We assume that $||\tilde{u}_{n}(\tau)-y_{0}||\leq R$ for
$n=0,\ldots,m$. It then follows from \eqref{cond2} and \eqref{recur}
that
\begin{equation*}
\begin{aligned}
&||\tilde{u}_{m+1}(\tau)-y_0||\leq   hA_0C_0D_0\leq R,\\
\end{aligned}
\end{equation*}
which means that $\tilde{u}_{n}(\tau)$ are uniformly bounded by $
||\tilde{u}_{n}(\tau)-y_{0}||\leq R$ for $n=0,1,\ldots.$ Then based
on  \eqref{recur}, we obtain \begin{equation*}
\begin{aligned}
&\ \ \ \ \|\tilde{u}_{n+1}(\tau)-\tilde{u}_{n}(\tau)\|_c\\
&\leq h\int_{0}^1 \int_{0}^\tau\Big\|\Big[  B(\tilde{u}_n(\alpha))
 P_{\alpha,\sigma}  \nabla H(\tilde{u}_n(\sigma))-  B(\tilde{u}_{n-1}(\alpha))
 P_{\alpha,\sigma}  \nabla H(\tilde{u}_{n-1}(\sigma))\Big]\Big\|_c d\alpha d \sigma\\
 &\leq h\int_{0}^1
\int_{0}^\tau\Big\|\Big[ B(\tilde{u}_n(\alpha))
 P_{\alpha,\sigma}  \nabla H(\tilde{u}_n(\sigma))-  B(\tilde{u}_n(\alpha))
 P_{\alpha,\sigma}  \nabla H(\tilde{u}_{n-1}(\sigma))\\
 &\ \ \ \ +  B(\tilde{u}_n(\alpha))
 P_{\alpha,\sigma}  \nabla H(\tilde{u}_{n-1}(\sigma))-  B(\tilde{u}_{n-1}(\alpha))
 P_{\alpha,\sigma}  \nabla H(\tilde{u}_{n-1}(\sigma))\Big]\Big\|_c d\alpha d \sigma\\
& \leq h(
A_0C_0D_{1}+A_0C_1D_{0})||\tilde{u}_{n}(\sigma)-\tilde{u}_{n-1}(\sigma)||
\leq\beta||\tilde{u}_{n}(\tau)-\tilde{u}_{n-1}(\tau)||_{c},\\
\end{aligned}
\end{equation*}
where  $\beta=\delta( A_0C_0D_{1}+A_0C_1D_{0})$ and  $
||w||_{c}=\max_{\tau\in[0,1]}||w(\tau)||$ for
  a continuous $\mathbb{R}^{d}$-valued function $w$ on
$[0,1]$.  Therefore, one arrives at
\begin{equation*}
||\tilde{u}_{n+1}-\tilde{u}_{n}||_{c}\leq\beta||\tilde{u}_{n}-\tilde{u}_{n-1}||_{c}
\end{equation*}
and
\begin{equation*}\label{recur3}
||\tilde{u}_{n+1}-\tilde{u}_{n}||_{c}\leq\beta^{n}||\tilde{u}_{1}-y_{0}||_{c}\leq\beta^{n}R,\quad
n=0,1,\ldots.
\end{equation*}
By Weierstrass $M$-test and the fact that $\beta<1,$  we  confirm
that
$\sum_{n=0}^{\infty}(\tilde{u}_{n+1}(\tau)-\tilde{u}_{n}(\tau))$ is
uniformly convergent.

If the  integrator    has another solution $\tilde{v}(\tau)$, then
the following inequalities are obtained
 {\begin{equation*}
||\tilde{u}(\tau)-\tilde{v}(\tau)|| \leq\beta
||\tilde{u}(\tau)-\tilde{v}(\tau)||\leq\beta||\tilde{u}-\tilde{v}||_{c},
\end{equation*}
and
\begin{equation*}
\norm{\tilde{u}-\tilde{v}}_{c}\leq\beta||\tilde{u}-\tilde{v}||_{c}.
\end{equation*}}
Therefore, we get  $||\tilde{u}-\tilde{v}||_{c}=0$ and
$\tilde{u}(\tau)\equiv \tilde{v}(\tau)$. Consequently, the solution
of  the  FFEP integrator \eqref{EDCr} exists and is unique.

In what follows, we prove    the result that  $\tilde{u}(\tau)$ is
smoothly dependent of $h$. This is true if the series
$\left\{\frac{\partial^{k}\tilde{u}_{n}}{\partial
h^{k}}(\tau)\right\}_{n=0}^{\infty}$  is  uniformly convergent for
$k\geq1$.
  Differentiating  \eqref{recur} with respect to $h$ yields
\begin{equation}\label{recur2}
\begin{aligned}
\frac{\partial \tilde{u}_{n+1}}{\partial h}(\tau)&=
\int_{0}^1\big(\int_{0}^\tau B(\tilde{u}_n(\alpha))
 P_{\alpha,\sigma}d\alpha\big) \nabla H(\tilde{u}_n(\sigma)) d \sigma\\
 &+h\int_{0}^1\big(\int_{0}^\tau B^{(1)}(\tilde{u}_n(\alpha))
\frac{\partial \tilde{u}_{n}(\alpha)}{\partial h}
 P_{\alpha,\sigma}d\alpha\big) \nabla H(\tilde{u}_n(\sigma)) d \sigma\\
  &+h\int_{0}^1\big(\int_{0}^\tau B(\tilde{u}_n(\alpha))
 \frac{\partial P_{\alpha,\sigma}}{\partial h}d\alpha\big) \nabla H(\tilde{u}_n(\sigma)) d \sigma\\
  &+h\int_{0}^1\big(\int_{0}^\tau B(\tilde{u}_n(\alpha))
 P_{\alpha,\sigma}d\alpha\big) \nabla H^{(1)}(\tilde{u}_n(\sigma)) \frac{\partial \tilde{u}_{n}(\sigma)}{\partial h}d \sigma. \end{aligned}
\end{equation}
Hence, we have
\begin{equation*}
\norm{\frac{\partial \tilde{u}_{n+1}}{\partial
h}}_{c}\leq\alpha+\beta\norm{\frac{\partial \tilde{u}_{n}}{\partial
h}}_{c}\quad \ \textmd{with}\ \ \alpha=A_0C_0D_0+\delta A_1C_0D_0,
\end{equation*}
which yields  that $\left\{\frac{\partial \tilde{u}_{n}}{\partial
h}(\tau)\right\}_{n=0}^{\infty}$ is uniformly bounded as follows:
\begin{equation*}\label{bound}
\norm{\frac{\partial \tilde{u}_{n}}{\partial
h}}_{c}\leq\alpha(1+\beta+\ldots+\beta^{n-1})\leq\frac{\alpha}{1-\beta}=C^{*},\quad
n=0,1,\ldots.
\end{equation*}
 It follows from \eqref{recur2}  that
\begin{equation*}
\begin{aligned}
 &\ \ \ \ \frac{\partial \tilde{u}_{n+1}}{\partial h}-\frac{\partial
\tilde{u}_{n}}{\partial h} = \int_{0}^1\int_{0}^\tau \Big[
B(\tilde{u}_n(\alpha))
 P_{\alpha,\sigma}  \nabla H(\tilde{u}_n(\sigma))-
B(\tilde{u}_{n-1}(\alpha))
 P_{\alpha,\sigma}  \nabla H(\tilde{u}_{n-1}(\sigma)) \Big]d\alpha d \sigma\\
 &+h\int_{0}^1\int_{0}^\tau\Big[ B^{(1)}(\tilde{u}_n(\alpha))
\frac{\partial \tilde{u}_{n}(\alpha)}{\partial h}
 P_{\alpha,\sigma}  \nabla H(\tilde{u}_n(\sigma))-  B^{(1)}(\tilde{u}_{n-1}(\alpha))
\frac{\partial \tilde{u}_{n-1}(\alpha)}{\partial h}
 P_{\alpha,\sigma}  \nabla H(\tilde{u}_{n-1}(\sigma)) \Big]d\alpha d \sigma\\
  &+h\int_{0}^1\int_{0}^\tau\Big[  B(\tilde{u}_n(\alpha))
 \frac{\partial P_{\alpha,\sigma}}{\partial h}  \nabla H(\tilde{u}_n(\sigma))-  B(\tilde{u}_{n-1}(\alpha))
 \frac{\partial P_{\alpha,\sigma}}{\partial h}  \nabla H(\tilde{u}_{n-1}(\sigma))\Big]d\alpha d \sigma\\
  &+h\int_{0}^1\int_{0}^\tau \Big[  B(\tilde{u}_n(\alpha))
 P_{\alpha,\sigma}  \nabla H^{(1)}(\tilde{u}_n(\sigma)) \frac{\partial \tilde{u}_{n}(\sigma)}{\partial h}-  B(\tilde{u}_{n-1}(\alpha))
 P_{\alpha,\sigma}  \nabla H^{(1)}(\tilde{u}_{n-1}(\sigma)) \frac{\partial \tilde{u}_{n-1}(\sigma)}{\partial h}\Big]d\alpha d
 \sigma.
\end{aligned}
\end{equation*}
By adding and removing some expressions and with careful
simplifications, we obtain
\begin{equation*}
\begin{aligned}
\norm{\frac{\partial \tilde{u}_{n+1}}{\partial h}-\frac{\partial
\tilde{u}_{n}}{\partial h}}_{c}\leq \gamma
\beta^{n-1}+\beta\norm{\frac{\partial \tilde{u}_{n}}{\partial
h}-\frac{\partial \tilde{u}_{n-1}}{\partial h}}_{c},
\end{aligned}
\end{equation*}
where
$$\gamma=(C_0A_0D_{1}+C_1A_0D_{0}+\delta C_0A_1D_{1}+\delta C_1A_1D_{0}+2 \delta C_1A_0D_{1}C^{*}
+ \delta  A_0D_{0}C^{*}M_2+ \delta C_0A_0C^{*}L_{2})R.$$ Here,
$L_{2}$ and $M_{2}$ are constants satisfying
\begin{equation*}\begin{aligned}
&||\nabla H^{(1)}(y)-\nabla H^{(1)}(z)||\leq
L_{2}||y-z||,\quad\text{for\ \ $y,z\in B(y_{0},R)$},\\
&||B^{(1)}(y)-B^{(1)}(z)||\leq M_{2}||y-z||,\qquad \  \text{for\ \
$y,z\in B(y_{0},R)$}.\end{aligned}
\end{equation*}
 Therefore,  by induction, it is true that
\begin{equation*}
\norm{\frac{\partial \tilde{u}_{n+1}}{\partial h}-\frac{\partial
\tilde{u}_{n}}{\partial h}}_{c}\leq
n\gamma\beta^{n-1}+\beta^{n}C^{*},\quad n=1,2,\ldots,
\end{equation*}
which confirms the uniform convergence of
 $\sum_{n=0}^{\infty}(\frac{\partial \tilde{u}_{n+1}}{\partial
h}(\tau)-\frac{\partial \tilde{u}_{n}}{\partial h}(\tau))$. Thus,
 $\left\{\frac{\partial \tilde{u}_{n}}{\partial
h}(\tau)\right\}_{n=0}^{\infty}$ is uniformly convergent.

 Similarly,  the uniform convergence  of  other function series
$\left\{\frac{\partial^{k}\tilde{u}_{n}}{\partial
h^{k}}(\tau)\right\}_{n=0}^{\infty}$ for $k\geq2$ can be shown  as
well. Therefore, $\tilde{u}(\tau)$ is  smoothly dependent on $h$.
\end{proof}

\section{Algebraic order}\label{sec order}
In this section, we study the algebraic order of the FFEP
integrator. To this end, we first need to show the regularity of the
integrators. Following \cite{Li_Wu(na2016)}, if an $h$-dependent
function $w(\tau)$ can be expanded as
\begin{equation*}
w(\tau)=\sum_{n=0}^{r-1}w^{[n]}(\tau)h^{n}+\mathcal{O}(h^{r}),
\end{equation*}
then $w(\tau)$  is called as regular, where
$w^{[n]}(\tau)=\frac{1}{n!}\frac{\partial^{n}w(\tau)}{\partial
h^{n}}|_{h=0}$ is a vector-valued function with polynomial entries
of degrees $\leq n$.

\begin{lem}\label{lemma1}
The FFEP  integrator \eqref{EDCr} gives a regular $h$-dependent
 function $\tilde{u}(\tau)$.
\end{lem}
\begin{proof}
It has been proved in Theorem \ref{eus} that $\tilde{u}(\tau)$ is
smoothly dependent  on $h$. Therefore, we can expand
$\tilde{u}(\tau)$ with respect to $h$ at zero as follows:
\begin{equation*}
\tilde{u}(\tau)=\sum_{m=0}^{r-1}\tilde{u}^{[m]}(\tau)h^{m}+\mathcal{O}(h^{r}).
\end{equation*}
Let $\delta=\tilde{u}(\sigma)-y_{0}$. We have
\begin{equation*}
\delta=
\tilde{u}^{[0]}(\sigma)-y_{0}+\mathcal{O}(h)=y_{0}-y_{0}+\mathcal{O}(h)=\mathcal{O}(h).
\end{equation*}
Expanding  $\nabla H(\tilde{u}(\sigma))$ at $y_{0}$ and inserting
the above equalities into  \eqref{method inte}  leads to
\begin{equation}\label{comp}
\begin{aligned} &
\sum_{m=0}^{r-1}\tilde{u}^{[m]}(\tau)h^{m} =y_0+h\int_{0}^1
\int_{0}^\tau P_{\alpha,\sigma} B(\tilde{u}(\alpha)) d\alpha
\sum_{n=0}^{r-1}\frac{1}{n!}\nabla H^{(n)}(y_{0})
 (\underbrace{\delta,\ldots,\delta}_{n-fold})d
 \sigma+\mathcal{O}(h^{r}).
\end{aligned}
\end{equation}

In what follows, we prove the following result by induction
$$\tilde{u}^{[m]}(\tau)\in
P_{m}^{d}=\underbrace{P_{m}([0,1])\times\ldots\times
P_{m}([0,1])}_{d-fold}\ \ \ \textmd{for}\  \ \ m=0,1,\ldots,r-1,$$
 where  $P_{m}([0,1])$ consists of polynomials of degrees $\leq m$ on
$[0,1]$.

Firstly, $\tilde{u}^{[0]}(\tau)=y_{0}\in P_{0}^{d}$. Assume that
$\tilde{u}^{[n]}(\tau)\in P_{n}^{d}$ for $n=0,1,\ldots,m$. Compare
the coefficients of $h^{m+1}$ on both sides of \eqref{comp} and then
we have
\begin{equation*}
\begin{aligned}
&\tilde{u}^{[m+1]}(\tau)= \sum_{k+n=m} \int_{0}^1 \int_{0}^\tau
\big[P_{\alpha,\sigma} B(\tilde{u}(\alpha))\big]^{[k]} d\alpha
 h_{n}(\sigma)d\sigma,\quad
h_{n}(\sigma)\in P_{n}^{d}.
\end{aligned}
\end{equation*}
Since $P_{\alpha,\sigma}$ is regular (see \cite{Li_Wu(na2016)}) and
$\tilde{u}^{[n]}(\tau)\in P_{n}^{d}$, it is easy to  verify that
$\big[P_{\alpha,\sigma} B(\tilde{u}(\alpha))\big]^{[k]}\in
P_{k}^{d\times d}$. Thus, under the condition $k+n=m$,  we have
$$\sum_{k+n=m} \int_{0}^1 \int_{0}^\tau \big[P_{\alpha,\sigma}
B(\tilde{u}(\alpha))\big]^{[k]} d\alpha
 h_{n}(\sigma)d\sigma\in
P^d_{m+1}.$$ Therefore, it is true  that
\begin{equation*}
\begin{aligned}
\tilde{u}^{[m+1]}(\tau)  \in P_{m+1}^{d}.\\
\end{aligned}
\end{equation*}
\end{proof}

The following result will be used in the analysis of algebraic
order.
\begin{lem} \label{lemma2}(See \cite{Li_Wu(na2016)})
Given a regular function $w$ and an $h$-independent  sufficiently
smooth  function $g$, the composition (if exists) is regular.
Moreover,  one has
\begin{equation*}
\mathcal{P}_{h}g(w(\tau))-g(w(\tau))=\mathcal{O}(h^{r}).
\end{equation*}
\end{lem}

Before giving the algebraic order of the integrators, we recall the
following elementary theory of ordinary differential equations.
Denoting by $y(\cdot,\tilde{t}, \tilde{y})$ the solution of
$y'(t)=B(y(t)) \nabla H(y(t))$ satisfying the initial condition
$y(\tilde{t},\tilde{t}, \tilde{y})=\tilde{y}$ for any given
$\tilde{t}\in[0,h]$ and setting
\begin{equation*}
\Phi(s,\tilde{t}, \tilde{y})=\frac{\partial y(s,\tilde{t},
\tilde{y})}{\partial \tilde{y}},
\label{Phi}%
\end{equation*}
one has the   standard result
\begin{equation*}
\frac{\partial y(s,\tilde{t}, \tilde{y})}{\partial
\tilde{t}}=-\Phi(s,\tilde{t}, \tilde{y}) B(\tilde{y}) \nabla
H(\tilde{y}).
\end{equation*}

\begin{theo}\label{order}
The FFEP integrator \eqref{EDCr} is of order $2r$, which implies
\begin{equation*} \begin{aligned}
&\tilde{u}(1)-y(t_{0}+h)=\mathcal{O}(h^{2r+1}).\end{aligned}
\end{equation*}
Moreover, we have
\begin{equation*} \begin{aligned}
&\tilde{u}(\tau)-y(t_{0}+\tau h)=\mathcal{O}(h^{r+1}),\ \ \
0<\tau<1.\end{aligned}
\end{equation*}
\end{theo}
\begin{proof}
 According to the previous preliminaries,
 we obtain
\begin{equation*}\label{stage}
\begin{aligned}
&\tilde{u}(1)-y(t_{0}+  h)=y(t_{0}+  h,t_{0}+  h,\tilde{u}(1))-y(t_{0}+  h,t_{0},y_{0})\\
&=\int_{0}^{1}\frac{d}{d\alpha}y(t_{0}+  h,t_{0}+\alpha h,\tilde{u}(\alpha))d\alpha\\
&=\int_{0}^{1}\Big(h\frac{\partial y}{\partial\tilde{t}}(t_{0}+
h,t_{0}+\alpha h,\tilde{u}(\alpha))+
\frac{\partial y}{\partial\tilde{y}}(t_{0}+  h,t_{0}+\alpha h,\tilde{u}(\alpha))h\tilde{u}^{\prime}(\alpha)\Big)d\alpha\\
&=\int_{0}^{1}\Big(-h\frac{\partial y}{\partial\tilde{y}}(t_{0}+
h,t_{0}+\alpha
h,\tilde{u}(\alpha))B(\tilde{u}(\alpha)) \nabla H(\tilde{u}(\alpha)) \\
&\quad+ \frac{\partial y}{\partial\tilde{y}}(t_{0}+ h,t_{0}+\alpha
h,\tilde{u}(\alpha))hB(\tilde{u}(\alpha)) \mathcal{P}_{h}\nabla H(\tilde{u}(\alpha))\Big)d\alpha\\
&=-h\int_{0}^{1}\Phi^{1}(\alpha)B(\tilde{u}(\alpha))\big(\nabla
H(\tilde{u}(\alpha))-\mathcal{P}_{h} \nabla H( \tilde{u} (\alpha))
\big)d\alpha,
\end{aligned}
\end{equation*}
where
$$\Phi^{1}(\alpha)=\frac{\partial y}{\partial\tilde{y}}(t_{0}+\ h,t_{0}+\alpha h,\tilde{u}(\alpha)).$$
From  Lemmas \ref{lemma1} and \ref{lemma2}, it follows that
\begin{equation*}\label{error}
\mathcal{P}_{h}\nabla H(\tilde{u}(\tau))-\nabla
H(\tilde{u}(\tau))=\mathcal{O}(h^{r}).
\end{equation*}
Partition the matrix-valued function $\Phi^{1}(\alpha)$  as
$\Phi^{1}(\alpha)=(\Phi_{1}^{1}(\alpha),\ldots,\Phi_{d}^{1}(\alpha))^{\intercal}$
and then it follows from Lemma \ref{lemma1} that
\begin{equation*}\label{pre1}
\Phi_{i}^{1}(\alpha)=\mathcal{P}_{h}\Phi_{i}^{1}(\alpha)+\mathcal{O}(h^{r}),\quad
i=1,2,\ldots,d.
\end{equation*}
Since $\mathcal{P}_{h}\Phi_{i}^{1}(\alpha)\in Y_h$, we have
\begin{equation*}
\begin{aligned}
&\int_{0}^{1}(\mathcal{P}_{h}\Phi_{i}^{1}(\alpha))^{\intercal}\nabla
H(\tilde{u}(\alpha))
d\alpha=\int_{0}^{1}(\mathcal{P}_{h}\Phi_{i}^{1}(\alpha))^{\intercal}\mathcal{P}_{h}
\nabla H( \tilde{u} (\alpha))d\alpha,\quad i=1,2,\ldots,d.
\end{aligned}\end{equation*}
 Therefore, one arrives at
\begin{equation*}
\begin{aligned}
&\tilde{u}(1)-y(t_{0}+h)
= -h\int_{0}^{1}\left(\left(\begin{array}{c}(\mathcal{P}_{h}\Phi_{1}^{1}(\alpha))^{\intercal}\\
\vdots\\(\mathcal{P}_{h}\Phi_{d}^{1}(\alpha))^{\intercal}\end{array}\right)
+\mathcal{O}(h^{r})\right)\big(\nabla
H(\tilde{u}(\alpha))-\mathcal{P}_{h} \nabla H( \tilde{u} (\alpha))\big)d\alpha\\
=&-h\int_{0}^{1}\left(\begin{array}{c}(\mathcal{P}_{h}\Phi_{1}^{1}(\alpha))^{\intercal}\big(\nabla
H(\tilde{u}(\alpha))-\mathcal{P}_{h} \nabla H( \tilde{u} (\alpha))\big)\\
\vdots\\(\mathcal{P}_{h}\Phi_{d}^{1}(\alpha))^{\intercal}\big(\nabla
H(\tilde{u}(\alpha))-\mathcal{P}_{h} \nabla H( \tilde{u}
(\alpha))\big)\end{array}\right)d\alpha
-h\int_{0}^{1}\mathcal{O}(h^{r})\times\mathcal{O}(h^{r})d\alpha\\
=&0+\mathcal{O}(h^{2r+1})=\mathcal{O}(h^{2r+1}).\\
\end{aligned}
\end{equation*}
 Likewise,  we deduce that
\begin{equation*}
\begin{aligned}
&\tilde{u}(\tau)-y(t_{0}+  \tau
h)=y(t_{0}+\tau h,t_{0}+\tau h,\tilde{u}(\tau))-y(t_{0}+\tau h,t_{0},y_{0})\\
&=-h\int_{0}^{\tau}\Phi^{\tau}(\alpha)B(\tilde{u}(\alpha))\big(\nabla
H(\tilde{u}(\alpha))-\mathcal{P}_{h} \nabla H( \tilde{u} (\alpha))
\big)d\alpha\\
&=-h\int_{0}^{\tau}\Phi^{\tau}(\alpha)B(\tilde{u}(\alpha))\mathcal{O}(h^{r})d\alpha=\mathcal{O}(h^{r+1}).
\end{aligned}
\end{equation*}
\end{proof}


\section{Practical  FFEP integrators}\label{methods}
In this  section,  we present two   illustrative examples of the new
  FFEP integrators.

\vskip2mm\textbf{Example 1.} Choose $$\varphi_k(t)=t^k,\ \ \
k=0,1,\cdots,r-1$$ for the function spaces $X$ and $Y$,  and then
one gets that
\begin{equation*}\label{l func}\hat{l}_{i}(\tau)=\prod_{j=1,j\neq i}^{r}\frac{\tau-\hat{d}_j}{\hat{d}_i-\hat{d}_j},\ \ \
i=1,2\ldots,r.\end{equation*}
  Using  the
Gram-Schmide process, we obtain the standard orthonormal basis
 of $Y_{h}$ as
\begin{equation*}
\hat{p}_j(t)=(-1)^j\sqrt{2j+1}\sum\limits_{k=0}^ {j}{j
\choose{k}}{j+k \choose{k}}(-t)^k,\qquad j=0,1,\ldots,r-1,\qquad
t\in[0,1],
\end{equation*}
which are the shifted Legendre polynomials   on $[0,1]$.  Therefore,
$P_{\tau,\sigma}$ can be determined  by
\begin{equation*}\label{Plimit} P_{\tau,\sigma}
=\sum_{i=0}^{r-1}\hat{p}_{i}(\tau)\hat{p}_{i}(\sigma).
\end{equation*}
In this situation, the FFEP integrator \eqref{EDCr} becomes  the EP
  method   given by Cohen and Hairer
\cite{Hairer2010}  and   Brugnano et al.  \cite{Brugnano2012}.

 As an example, if we choose $r=1$ and $\hat{d}_{1}=1/2$, one has
$$\hat{l}_{1}(\tau)=1,\ \hat{p}_0(t)=1, $$
and $P_{\tau,\sigma}=1.$ The integrator \eqref{PEEPCr-2} becomes
\begin{equation}\label{m1}
\left\{\begin{aligned}& y_{\tau} =y_0+h \tau B(y_{\frac{1}{2}})
\int_{0}^1
 \nabla H(y_\sigma)d
\sigma,\\
& y_{1} =y_0+h  B(y_{\frac{1}{2}})\int_{0}^1   \nabla H(y_\sigma) d
\sigma,\\
\end{aligned}\right.
\end{equation}
 which leads to
\begin{equation*}
 \begin{aligned}y_{\tau} =y_0+h \tau B(y_{\frac{1}{2}})
\int_{0}^1
 \nabla H(y_\sigma)d
\sigma=y_0+\tau (y_1-y_0).
\end{aligned}
\end{equation*}
Letting $\tau=1/2$ for the first formula of \eqref{m1}  gives
\begin{equation*}
 \begin{aligned}& y_{\frac{1}{2}} =y_0+\frac{1}{2}h  B(y_{\frac{1}{2}}) \int_{0}^1
 \nabla H(\tilde{u}(\sigma)d
 \sigma=y_0+\frac{y_1-y_0}{2}=\frac{y_1+y_0}{2}.
\end{aligned}
\end{equation*}
Thus, we obtain
\begin{equation*}
 \begin{aligned}y_{1} =y_0+h  B(\frac{y_1+y_0}{2})\int_{0}^1   \nabla H(y_0+\sigma (y_1-y_0)) d
\sigma.
\end{aligned}
\end{equation*}
This second-order integrator has been given by  Cohen and Hairer in
\cite{Cohen-2011}.

\vskip2mm\textbf{Example  2.} We consider   another choice for $Y$
by
\begin{equation*}\label{FF} Y=\text{span}\left\{ \cos(\omega
t)\right\},
\end{equation*}  and then we get
\begin{equation*}\begin{aligned}\label{l func}&\hat{l}_{1}(\tau)= \frac{\cos(t v )}{\cos(\hat{d}_1v)},\
\ \ P_{\tau,\sigma}=\frac{4 v \cos(v \sigma) \cos(v \tau)}{2 v +
\sin(2v)},
\end{aligned}\end{equation*}
 where $v=\omega h.$
Under this choice, the integrator \eqref{PEEPCr-2} becomes
  \begin{equation*}
\left\{\begin{aligned}& y_{\tau} =y_0+h \int_{0}^{\tau}
\hat{l}_{1}(\alpha)d\alpha B(y_{\hat{d}_{1}}) \int_{0}^1
P_{\hat{d}_{1},\sigma} \nabla H(y_\sigma) d
\sigma,\\
& y_{1} =y_0+h \int_{0}^1 \hat{l}_{1}(\alpha)d\alpha
B(y_{\hat{d}_{1}})\int_{0}^1 P_{\hat{d}_{1},\sigma} \nabla
H(y_\sigma) d
\sigma.\\
\end{aligned}\right.
\end{equation*}
 The choice of  $\tau=\hat{d}_{1}=1/2$ yields
\begin{equation*}
 \begin{aligned} y_{1/2} =y_0+h \frac{\tan(v/2 )}{v} B(y_{1/2}) \int_{0}^1 P_{1/2,\sigma}
\nabla H(y_\sigma) d \sigma,\\
y_{1} =y_0+h \frac{2\sin(v/2 )}{v} B(y_{1/2}) \int_{0}^1
P_{1/2,\sigma} \nabla H(y_\sigma) d \sigma.
\end{aligned}\end{equation*} Therefore, using these two formulae,
we obtain
\begin{equation*}
 \begin{aligned}&y_{1/2} =y_0+  \frac{\tan(v/2 )}{v}\frac{v(y_{1}-y_{0})}{2\sin(v/2 )}=y_0+  \frac{1}{2\cos(v/2 )}
 (y_{1}-y_{0}),\\
&y_{\tau} =y_0+   \frac{\sin(v \tau )}{v\cos(v/2
)}\frac{v(y_{1}-y_{0})}{2\sin(v/2 )}=y_0+   \frac{\sin(v \tau
)}{\sin(v)}  (y_{1}-y_{0}),
\end{aligned}\end{equation*} which  leads to
\begin{equation*}
 \begin{aligned}y_{1} =&y_0+h \frac{2\sin(v/2 )}{v} B\Big(y_0+  \frac{y_{1}-y_{0}}{2\cos(v/2 )}
 \Big) \int_{0}^1 P_{1/2,\sigma} \nabla H\Big(y_0+   \frac{\sin(v
\sigma )}{\sin(v)} (y_{1}-y_{0})\Big) d \sigma.
\end{aligned}
\end{equation*}
It can be observed that when $v=0$, this scheme reduces to
\eqref{EXAVF}. This second-order  integrator is denoted by FFEP1.

\begin{rem}
It is noted that one can make  different choices of  $Y$ and $X$ and
  different practical integrators can be  derived. We do not  pursue  further on this point   for brevity.
\end{rem}

\section{Numerical experiments}\label{Numerical experiments}
In order to  show the efficiency and robustness of the new
integrators, we apply our integrator FFEP1 to the Euler equation.
For comparison, we consider the second-order EP collocation method
\eqref{EXAVF} given in \cite{Cohen-2011} and denote it by EPCM1.
Moreover, we choose the
 following second-order  trigonometrically-fitted
EP method \begin{equation*}
 \begin{aligned}y_{1} =y_0+h \frac{2\sinh(v/2 )}{v\cosh(v/2 )} B\big(\frac{y_1+y_0}{2}\big)\int_{0}^1   \nabla H(y_0+\sigma (y_1-y_0)) d
\sigma,
\end{aligned}
\end{equation*}
which was given in   \cite{Miyatake2015}. We denote it by TFEP1. It
is noted  that these three methods are  all implicit and fixed-point
iteration will be used. We
  set  $10^{-16}$  as the error tolerance
 and $10$ as the maximum number of each iteration.

 The following Euler equation  has been considered in
\cite{Calvo2009,Miyatake2015}:
\begin{equation*}
 \begin{aligned}\dot{y}=\big((\alpha-\beta)y_2y_3,(1-\alpha)y_3y_1,(\beta-1)y_1y_2\big)^{\intercal},
 \ \ t\in[0,T],
\end{aligned}
\end{equation*}
 which describes the motion of a rigid
body under no forces. This system can be  written  as a Poisson
system
\begin{equation*}
 \begin{aligned}\dot{y}= \left(
                           \begin{array}{ccc}
                             0 & \alpha y_3 & -\beta y_2 \\
                             - \alpha y_3 & 0 & y_1 \\
                             \beta y_2 & y_1 & 0 \\
                           \end{array}
                         \right)\nabla H(y)
\end{aligned}
\end{equation*}
with
$$H(y)=\frac{y_1^2+y_2^2+y_3^2}{2}.$$
Following \cite{Calvo2009,Miyatake2015}, the initial value is chosen
as $y(0) = (0, 1, 1)$, and the parameters  are given by
$\alpha=1+\frac{1}{\sqrt{1.51}},\ \beta=1-\frac{0.51}{\sqrt{1.51}}$.
The exact solution is given by
$$y(t) = (\sqrt{1.51}\textmd{sn}(t, 0.51), \textmd{cn}(t, 0.51), \textmd{dn}(t, 0.51))^{\intercal},$$ where
$\textmd{sn}, \textmd{cn}, \textmd{dn}$ are the Jacobi elliptic
functions. This solution is periodic with the period $$T_{p} =
7.450563209330954,$$ and thence we consider choosing
$\omega=2\pi/T_{p}$ for the methods FFEP1 and TFEP1. We integrate
this problem with
 the stepsizes $h=0.5$ and  $h=0.2$ in the interval $[0, 10000].$ See Figure
\ref{p1-2}  for the energy conservation for different methods. We
then solve the problem in the interval $[0, T]$ with different
stepsizes $h= 0.1/2^{i}$ for $i=4,5,6,7.$  The   global errors are
presented in Figure \ref{p1-1} for $T=10,100$.


\begin{figure}[ptb]
\centering
\includegraphics[width=6cm,height=8cm]{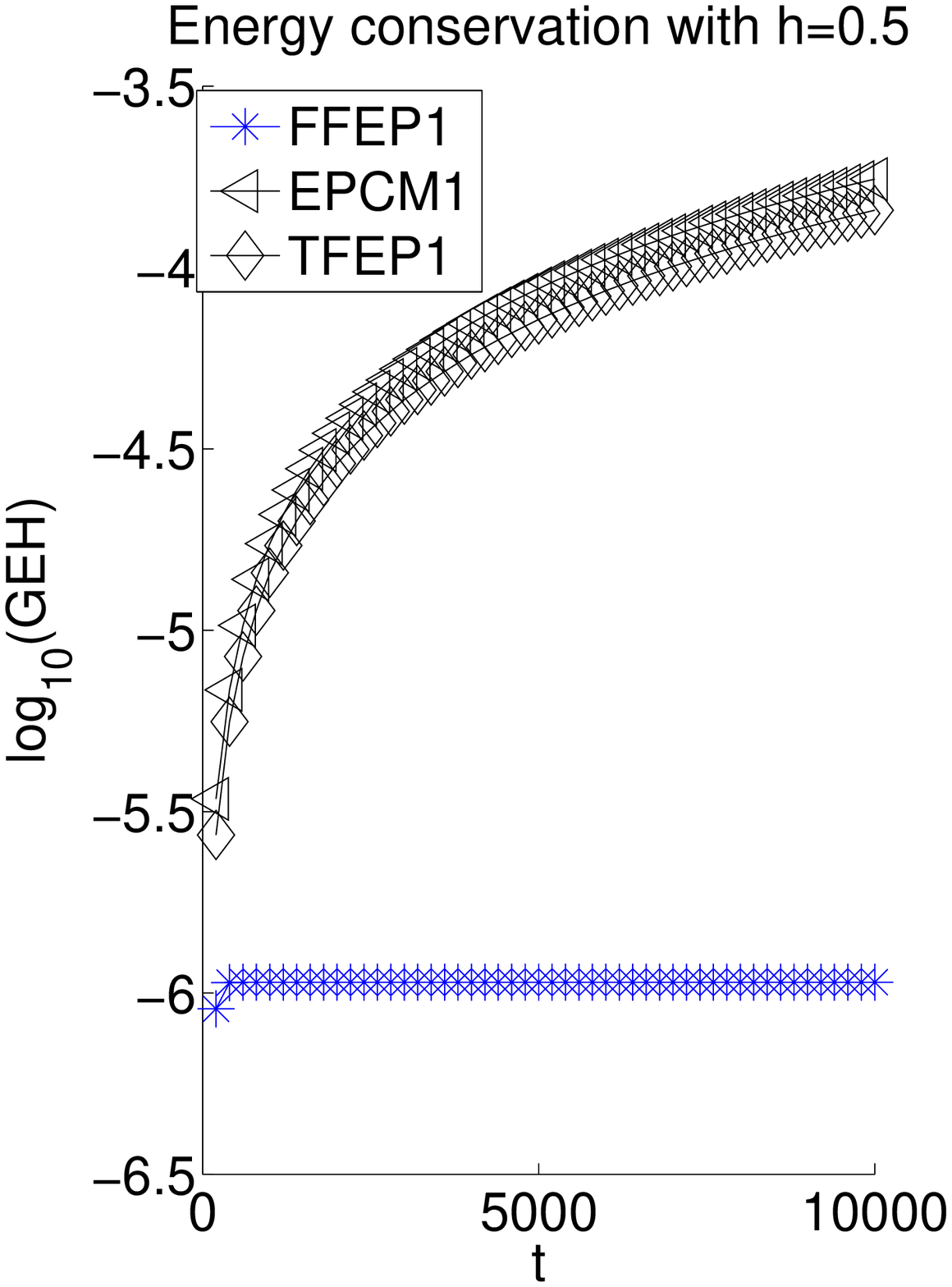}
\includegraphics[width=6cm,height=8cm]{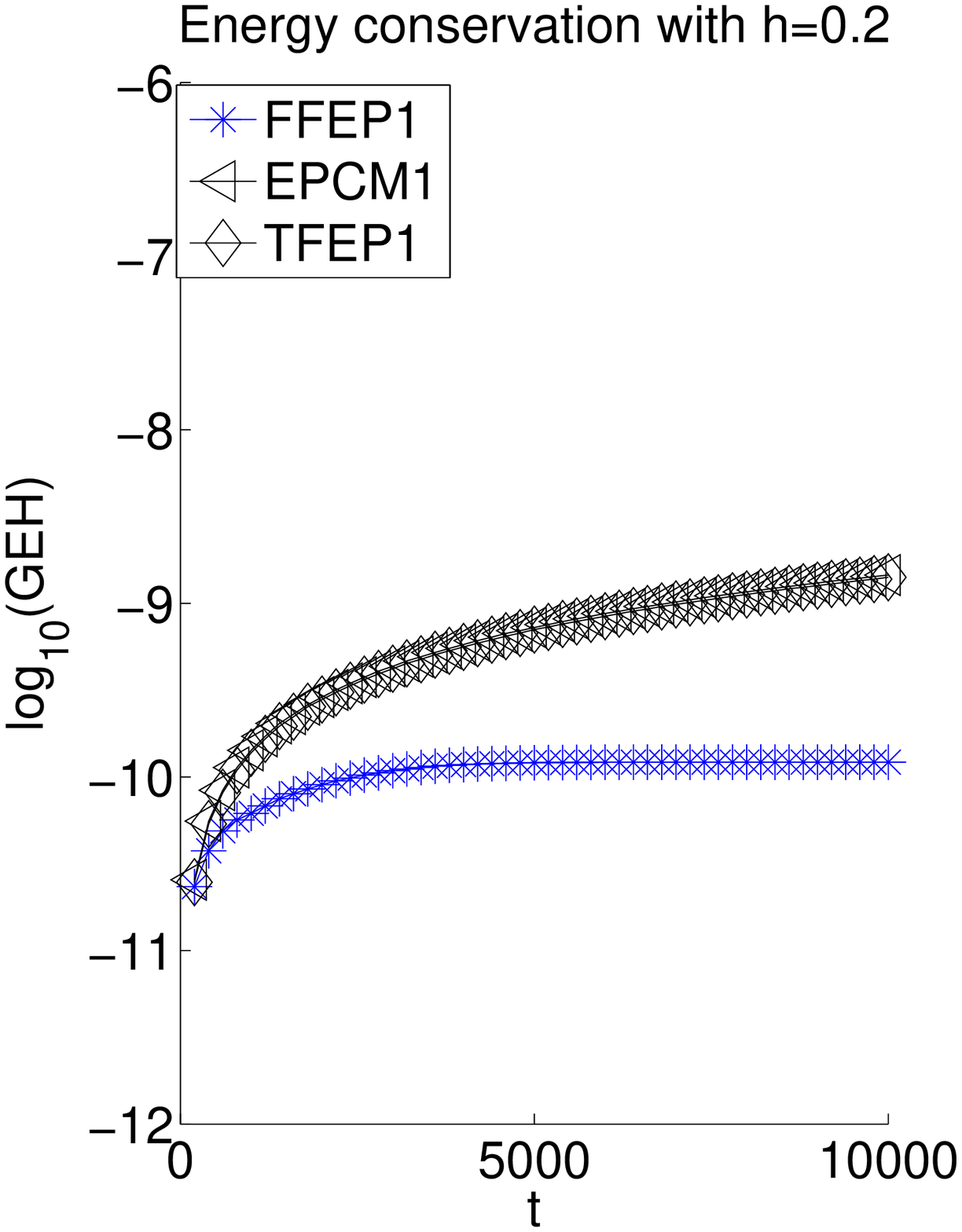}
\caption{  The logarithm of the  error of Hamiltonian against  $t$.
} \label{p1-2}
\end{figure}

 \begin{figure}[ptb]
\centering
\includegraphics[width=6cm,height=8cm]{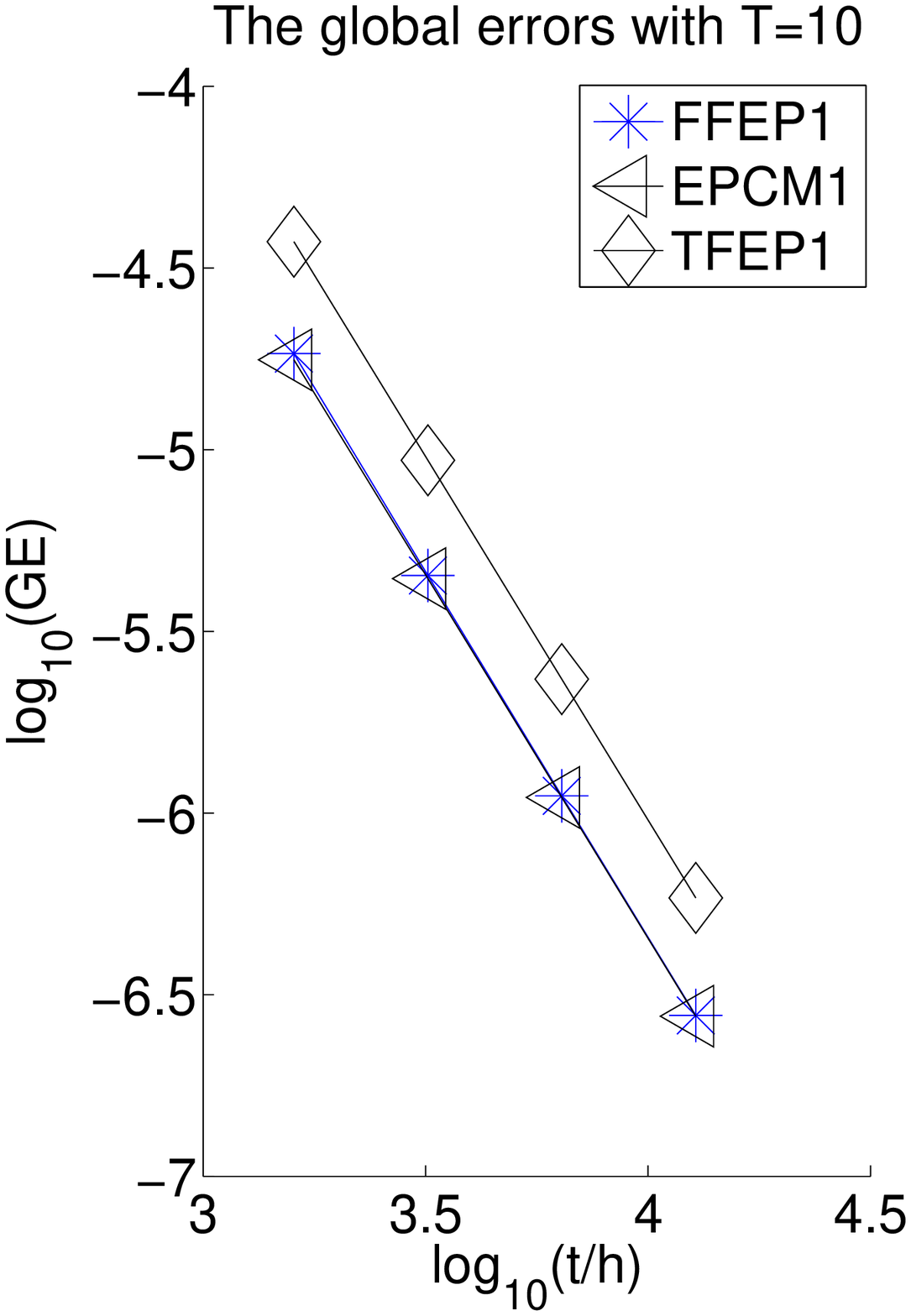}
\includegraphics[width=6cm,height=8cm]{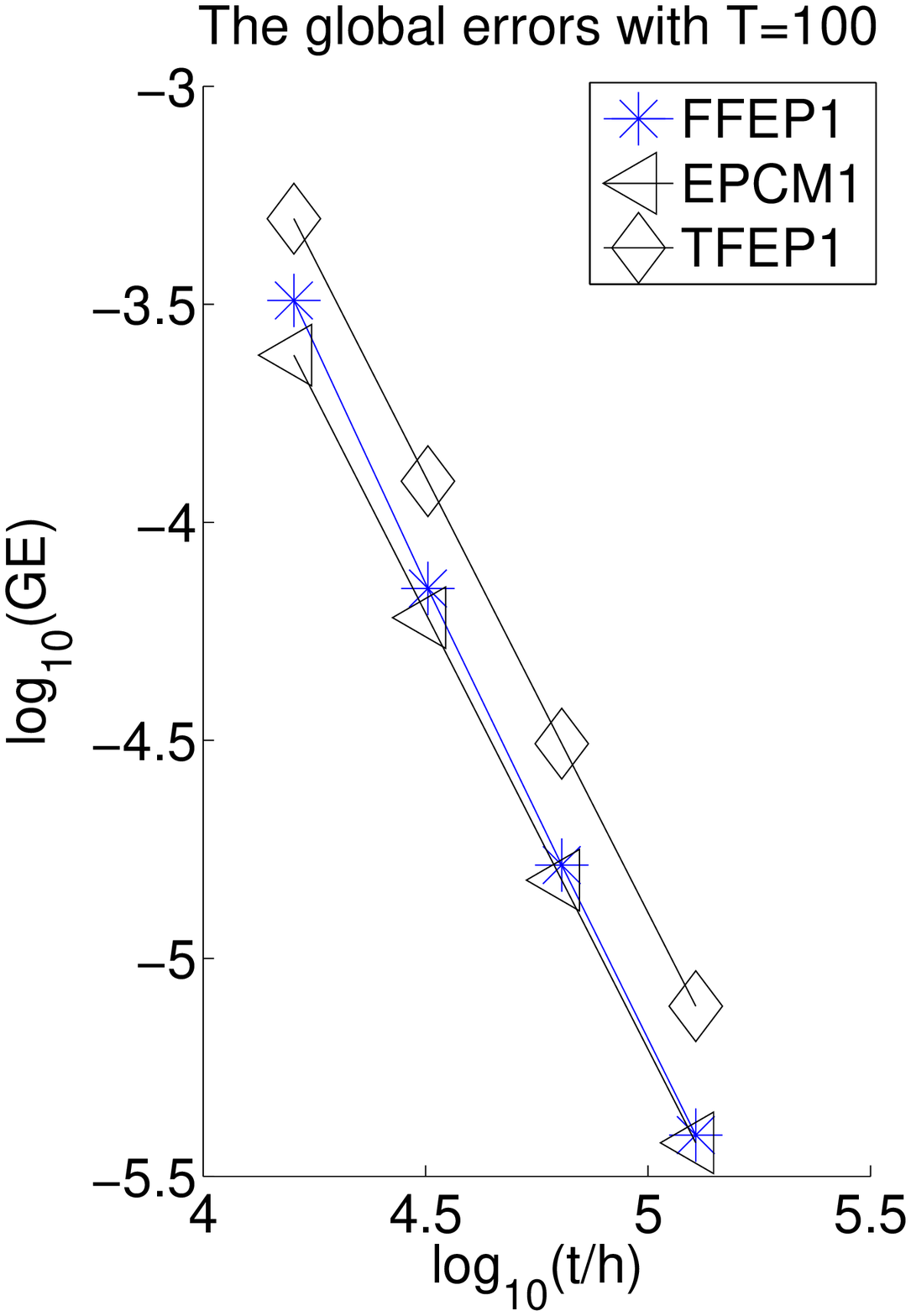}
\caption{ The logarithm of the  global error against the logarithm
of $t/h$.  } \label{p1-1}
\end{figure}

 We also consider a more anomalous case. As mentioned in
\cite{Miyatake2015},  when $\beta\approx 1$, it is  expected that
$\dot{y}_3\approx 0$ and thus $y_3(t) \approx 1.$ Therefore, the
variables $y_1$ and $y_2$ seem to behave like harmonic oscillator
with  the  period $T_{p} = 2\pi/(\alpha- 1).$ We choose $\alpha =
51$ and $\beta = 1.01$, which means that $\omega=50$.  We integrate
this problem with  $h=0.5$ and  $h=0.2$ in the interval $[0,
10000].$ The energy conservation for different methods are shown in
 Figure \ref{p1-4}. Then the problem is solved in the interval $[0, T]$ with
 $h= 0.1/2^{i}$ for $i=4,5,6,7,$ and see Figure \ref{p1-3} for the   global errors of $T=10,20$.

\begin{figure}[ptb]
\centering
\includegraphics[width=6cm,height=8cm]{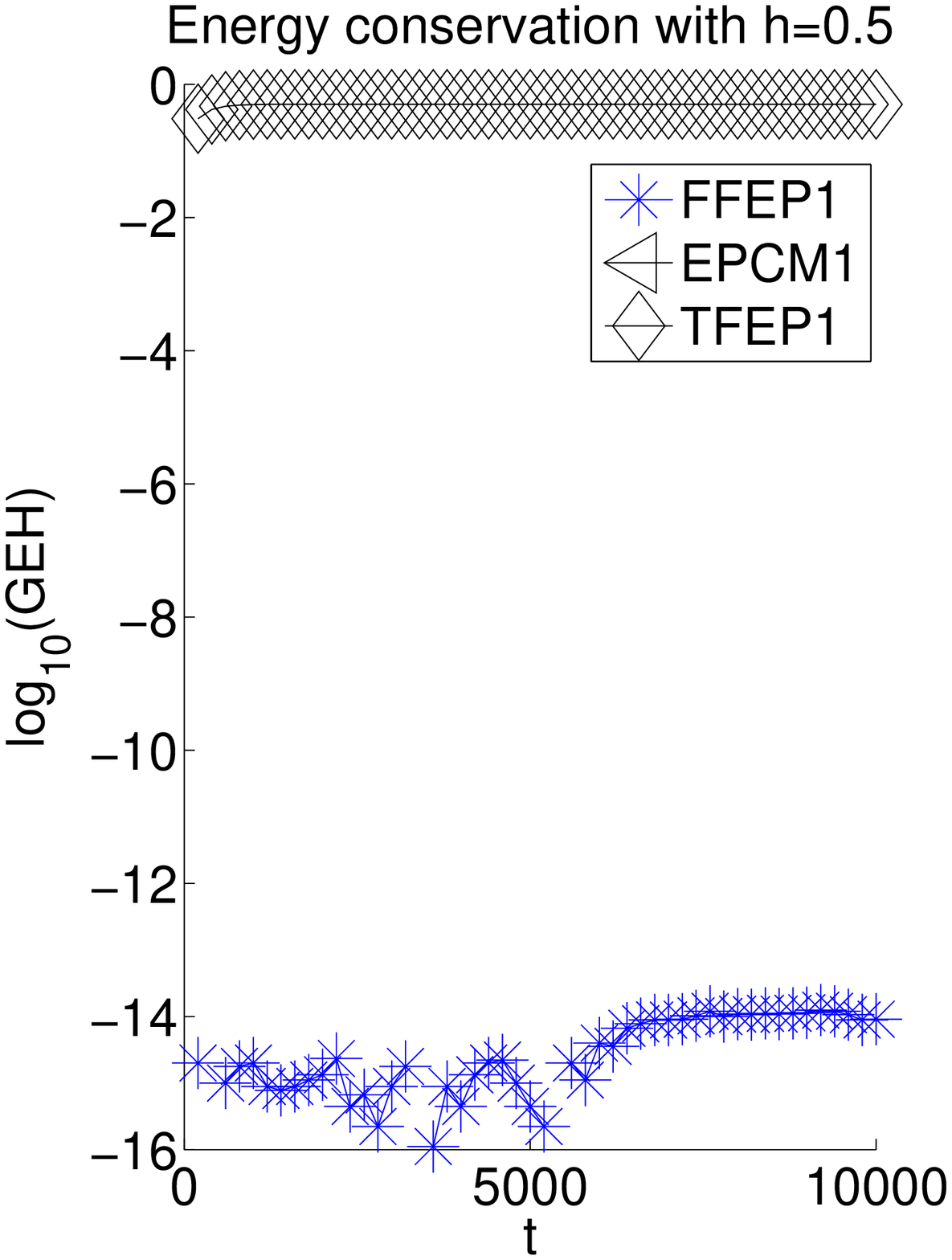}
\includegraphics[width=6cm,height=8cm]{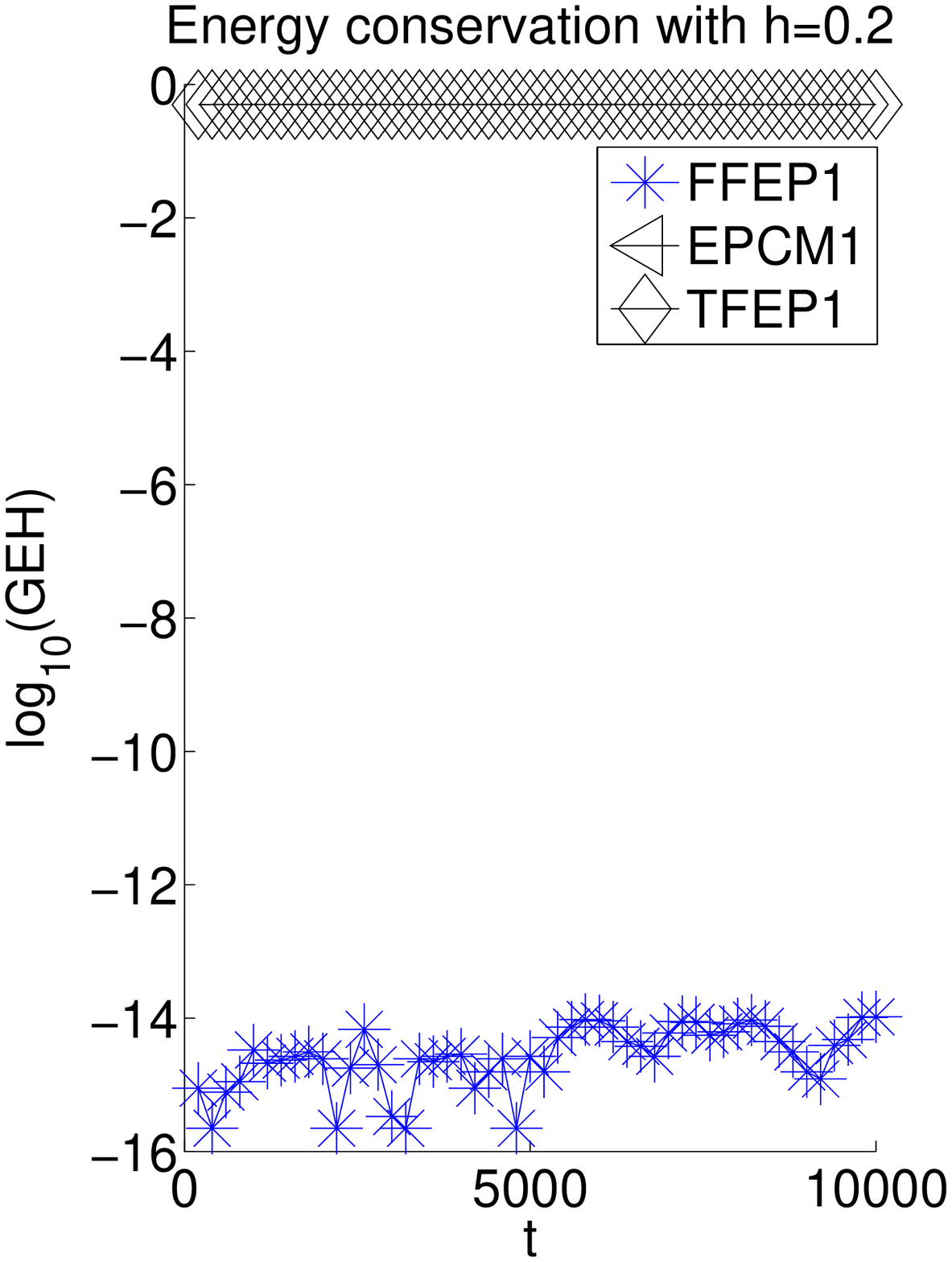}
\caption{  The logarithm of the  error of Hamiltonian against  $t$.
} \label{p1-4}
\end{figure}

 \begin{figure}[ptb]
 \centering
\includegraphics[width=6cm,height=8cm]{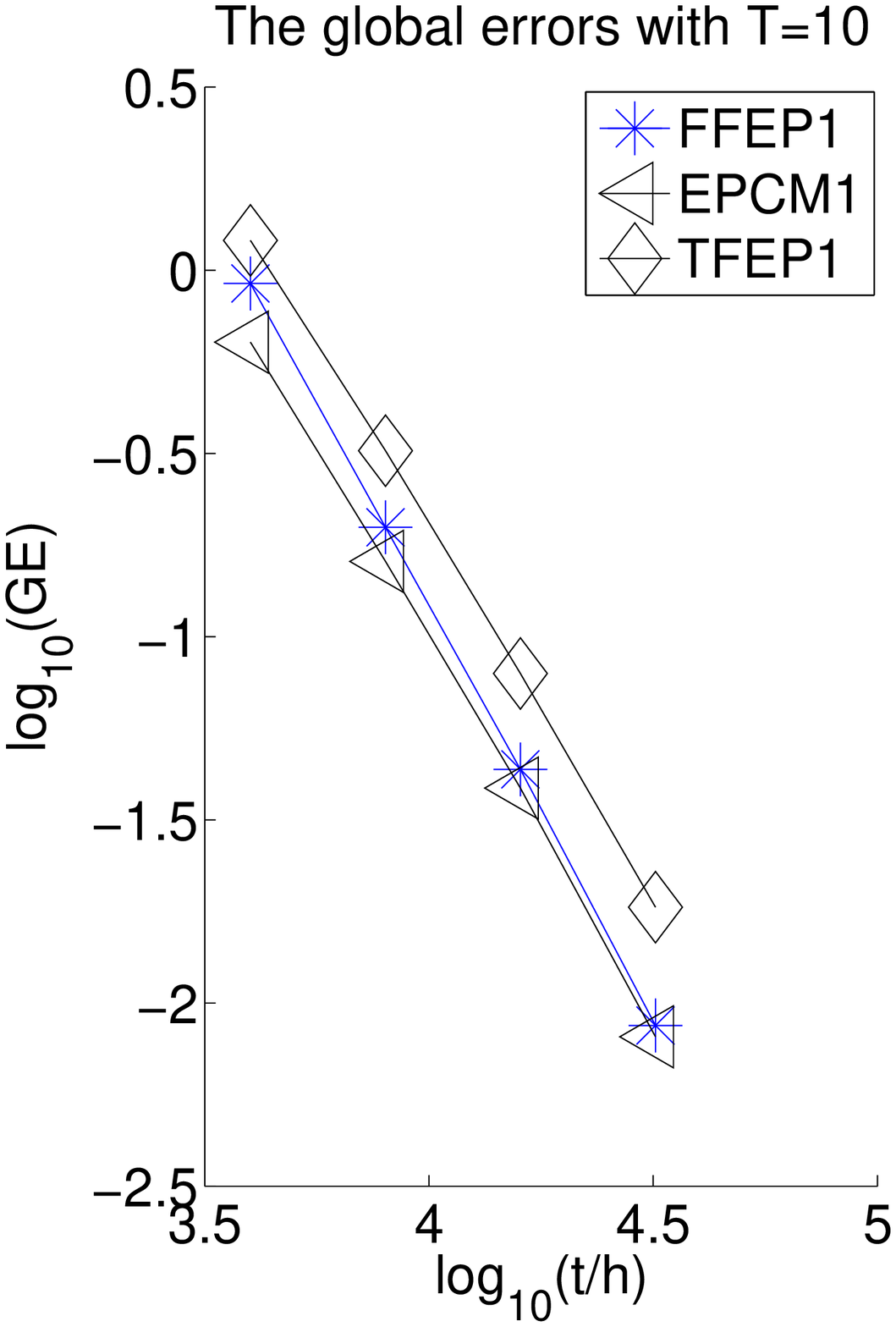}
\includegraphics[width=6cm,height=8cm]{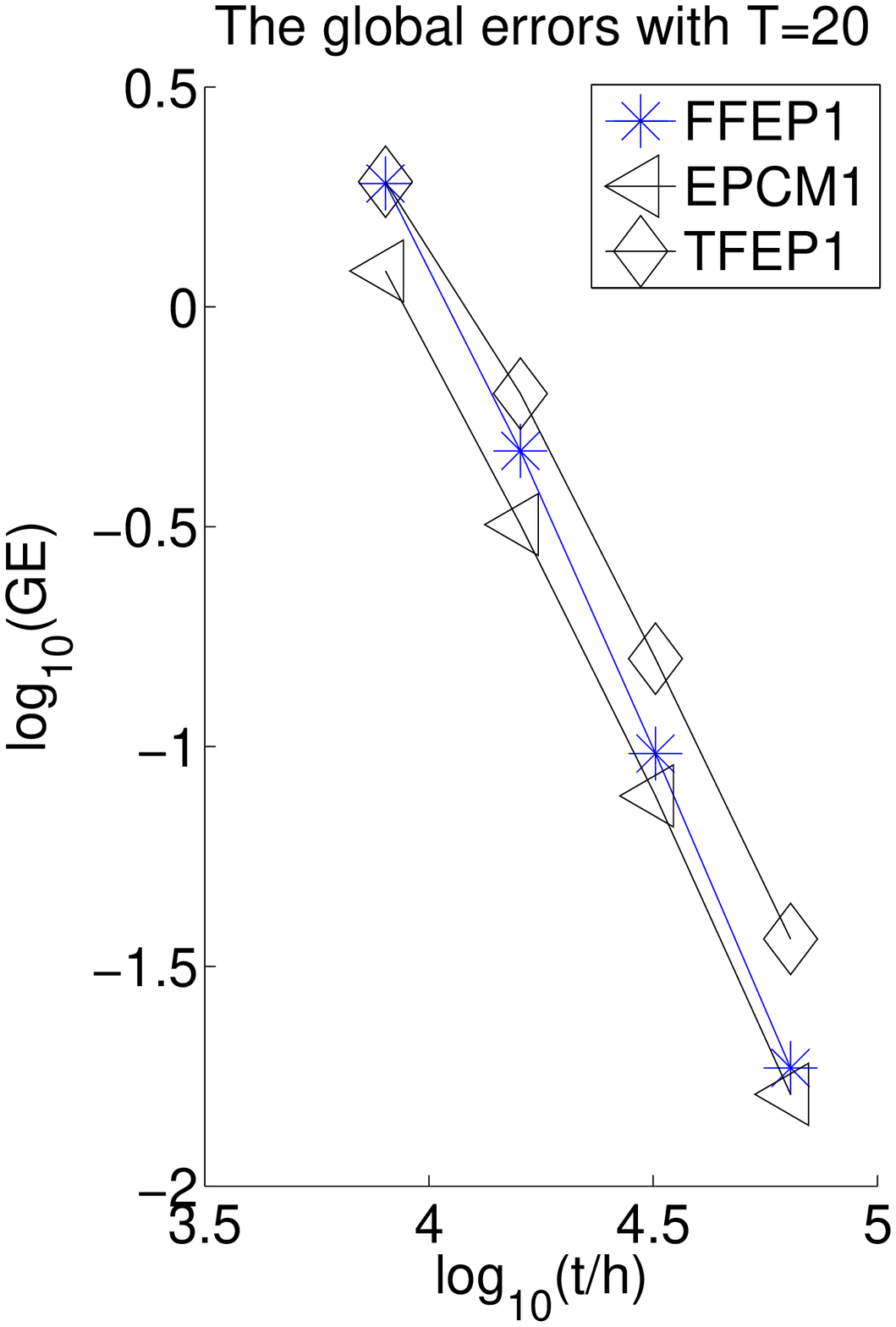}
\caption{ The logarithm of the  global error against the logarithm
of $t/h$.  } \label{p1-3}
\end{figure}

  It can be concluded from the numerical results  that
our FFEP1 method when applied to the underlying Euler equation shows
remarkable numerical behaviour in comparison with the existing EP
methods in the literature.

\section{Conclusions} \label{sec:conclusions}

In this paper, we derived and analysed functionally-fitted
energy-preserving integrators for Poisson systems by using
functionally-fitted technology. It has been shown that the novel
integrators preserve exactly the energy of Poisson systems and can
be of arbitrary-order  in a   convenient manner. The new integrators
contain the energy-preserving  schemes  given by Cohen and Hairer
\cite{Cohen-2011} and Brugnano et al. \cite{Brugnano2012}. The
remarkable efficiency  and robustness  of the integrators were
demonstrated through the numerical experiments for the Euler
equation. Our future work will be focused on developing
functionally-fitted energy-preserving integrators  for gradient
systems. We are hopeful of obtaining some new results within this
framework.

\end{document}